\documentclass[12pt]{amsart}
\usepackage{amsfonts, amsmath, amsthm, amstext, amssymb, amscd, mathrsfs, verbatim, color, enumerate, lscape}
\usepackage{hhline}

\newtheorem{theorem}{Theorem}[section]

\newtheorem{proposition}[theorem]{Proposition}
\newtheorem{lemma}[theorem]{Lemma}

\theoremstyle{definition}
\newtheorem{proof*}{Proof.}
\numberwithin{equation}{section}
\newtheorem*{remark}{Remark} 

\newtheorem{thmx}{Theorem}

\newcommand{\PG}{{\mbox{\rm PG}}}
\newcommand{\Q}{{\mathcal Q}}
\newcommand{\w}[1]{\widehat{#1}}
\newcommand{\wt}{{\mbox{\rm wt}}}
\newcommand{\Aut}{{\mbox{\rm Aut}}}

\begin{document}
\title{Switched graphs of some strongly regular graphs related to the symplectic graph}
\author{Alice M.W. Hui, Bernardo Rodrigues}
\maketitle

\centerline {\tiny huimanwa@gmail.com, Rodrigues@ukzn.ac.za, \today}

\begin{abstract}
Applying a method of Godsil and McKay \cite{GM} to some graphs related to the symplectic graph, a series of new infinite families of strongly regular graphs with parameters $(2^n\pm2^{(n-1)/2},2^{n-1}\pm2^{(n-1)/2},2^{n-2}\pm2^{(n-3)/2},2^{n-2}\pm2^{(n-1)/2})$
are constructed for any odd $n \geq 5$. The construction is described in terms of geometry of quadric in projective space. The binary linear codes of the switched graphs are $[2^n \mp 2^{\frac{n-1}{2}},n+3,2^{t+1}]_2$-code or $[2^n \mp 2^{\frac{n-1}{2}},n+3,2^{t+2}]_2$-code.
\end{abstract}

Keywords: strongly regular graph, cospectral graphs, linear code of a graph\\

MSC 2010: 05E30 05C50 94B25 51A50


\section{Introduction}\label{sc intro}
Consider the $n$-dimensional projective space $\PG(n,2)$ over the finite field $\mathbb F_2$. That is, $\PG(n,2)=\mathbb F_2^{n+1}\setminus \{ 0 \}$. When $n$ is odd, there are two non-equivalent non-singular quadrics in $\PG(n,2)$, namely elliptic and hyperbolic. For general references, see \cite[Ch. 5]{Hir1} and \cite[Ch. 22]{HirT3}. Both quadrics define a symplectic polarity (null polarity) in $\PG(n,2)$ \cite[Thoerem 5.28]{Hir1}.

Let $n \geq 5$ be an odd number. Let $\Q$ be a non-singular quadric in $\PG(n,2)$. Define the graph $\Gamma_{\Q}=(V_\Q,E_\Q)$ as follows. The vertex set $V_\Q$ is the set of points of $\PG(n,2)$ not in $\Q$. Two vertices $x$ and $y$ are adjacent in $\Gamma_{\Q}$ if and only if the line $xy$ joining them is an external line of $\Q$. $\Gamma_{\Q}$ is the complement of a subgraph of the symplectic graph $Sp(n+1,2)$, which is the graph of the perpendicular relation induced by a non-degenerate symplectic form of $\mathbb F_2^{n+1}$ on the non-zero vectors of $\mathbb F_2^{n+1}$. In \cite{HS,HPV,P}, $\Gamma_{\Q}$ is denoted by $\overline{\mathcal N^\epsilon_{n+1}}$, where $\epsilon$ is $+$ (plus) if $\Q$ is hyperbolic, and $-$ (minus) if $\Q$ is elliptic.

A {\it strongly regular graph} with parameters $(v,k,\lambda,\mu)$ is a graph with $v$ vertices such that each vertex lies on exactly $k$ edges; any two adjacent vertices have exactly $\lambda$ neighbours in common; and any two non-adjacent vertices have exactly $\mu$ neighbours in common. The adjacency matrix of a strongly regular graph has exactly three eigenvalues. One is $k$ with multiplicity 1, and the remaining two are usually denoted by $r$ and $s$, $r>s$ with multiplicities $f$ and $g$ respectively. For general references, see \cite[Ch.9]{BH} and \cite[Ch.2]{CV}. It is well-known that $\Gamma_{\Q}$ defined above is a strongly regular graph. Table \ref{table para} shows the parameters of $\Gamma_{\Q}$ for the different quadrics in $\PG(n,2)$ (see \cite{HPV}).

\begin{table}[h!]
\begin{tabular}{|c|c||c|c|c|c|}
 \hline
 $\Q$ & graph & $v$ & $k$ & $\lambda$ & $\mu$ \\
 \hline
 elliptic & $\Gamma_{\Q}=\overline{\mathcal N^-_{n+1}}$ & $2^n+2^{\frac{n-1}{2}}$ & $2^{n-1}+2^{\frac{n-1}{2}}$ & $2^{n-2}+2^{\frac{n-3}{2}}$ & $2^{n-2}+2^{\frac{n-1}{2}}$\\
 hyperbolic & $\Gamma_{\Q}=\overline{\mathcal N^+_{n+1}}$ & $2^n-2^{\frac{n-1}{2}}$ & $2^{n-1}-2^{\frac{n-1}{2}}$ & $2^{n-2}-2^{\frac{n-3}{2}}$ & $2^{n-2}-2^{\frac{n-1}{2}} $ \\
 \hline
 \hline
 $\Q$ & graph & $r$ & $s$ & $f$ &$g$\\
 \hline
 elliptic & $\Gamma_{\Q}=\overline{\mathcal N^-_{n+1}}$ & $2^{\frac{n-3}{2}}$& $-2^{\frac{n-1}{2}}$ & $\frac{1}{3}(2^{n+1}-4)$ & $\frac{2^n+1}{3}+2^{\frac{n-1}{2}}$ \\
 hyperbolic & $\Gamma_{\Q}=\overline{\mathcal N^+_{n+1}}$ & $2^{\frac{n-1}{2}}$ & $-2^{\frac{n-3}{2}}$ & $\frac{2^n+1}{3}-2^{\frac{n-1}{2}}$ & $\frac{1}{3}(2^{n+1}-4)$\\
 \hline
\end{tabular}
\vspace{0.3cm}
\caption{Parameters of $\Gamma_{\Q}$}\label{table para}
\end{table}
\vspace{-0.3cm}

Godsil and McKay (1982) introduced a method to generate graphs with the same adjacency spectrum \cite{GM} i.e. the adjacency matrices of the graphs have equal multisets of eigenvalues. The method is described as follows. Let $\Gamma$ be a graph. Let $S$ be a subset of the vertex set such that the subgraph of $\Gamma$ with vertex set $S$ is regular. Suppose any vertex outside $S$ has $0$, $|S|$ or $\frac{1}{2}|S|$ neighbours in $S$. Consider the graph $\Gamma'$ obtained by switching $\Gamma$ as follows: for any vertex $x$ of $\Gamma$ outside $S$, if $x$ has $\frac{1}{2}|S|$ neighbours in $S$, then delete those $\frac{1}{2}|S|$ edges and join $x$ to the other $\frac{1}{2}|S|$ vertices. We call $S$ a {\it Godsil and McKay switching set} of $\Gamma$. By Godsil and McKay \cite{GM}, $\Gamma'$ has the same adjacency spectrum as $\Gamma$. In the case where $\Gamma$ is a strongly regular graph, $\Gamma'$ has the same adjacency spectrum as $\Gamma$ and thus is also a strongly regular graph with the same parameters (see \cite{BH}). Recently, there has been interest in constructing new strongly regular graphs from known ones using the method of Godsil-McKay described above, see for example \cite{AH} and \cite{BJP}.

In this article, we apply the method of Godsil-McKay to $\Gamma_{\Q}$ as described above. The paper is organized as follows: After a brief description of our terminology in Section \ref{sc terminology}, we give two constructions of Godsil-McKay switching sets for $\Gamma_{\Q}$ in Section \ref{sc switching set}. In Sections \ref{sc codes} and \ref{sc rank}, we study the binary code spanned by the rows of the adjacency matrix $\Gamma_{\Q}$ and that of its switched graphs. In Section \ref{sc number}, we give a number of switched graphs found and find the parameters of the codes of the switched graphs.

\section{Terminology and notation}\label{sc terminology}

For any $m = 0,1,2,\cdots, n-1$, a {\it subspace of dimension $m$}, or {\it $m$-space}, of $\PG(n,2)$ is a set of points all of whose representing vectors form, together with the zero, a subspace of dimension $m+1$ of $\mathbb F_2^{n+1}$.
The number of points of an $m$-space in $\PG(n,2)$ is $2^{m+1}-1$ \cite[Theorem 3.1]{Hir1}.

A {\it quadric} $Q_n$ in $\PG(n,2)$ is the set of points $[X_0, X_1, \cdots , X_n]$ satisfying a non-zero homogeneous equation of degree two, i.e. $\sum_{i\leq j, i,j=0}^{n} a_{ij} X_i X_j = 0$ for some $a_{ij} \in \mathbb F_q$, not all zero. If the equation can be reduced to fewer than $n+1$ variables by a change of basis, $Q_n$ is called {\it singular}. Otherwise, it is {\it non-singular}.

Depending on the parity of $n$, there is one or there are two quadrics under the action of the automorphism group of $\PG(n,2)$. For $n$ odd, there are two distinct non-singular quadrics, respectively the elliptic quadric with canonical equation $f (X_0, X_1)+X_2X_3+\cdots+ X_{n-1}X_n=0$ where $f$ is an irreducible binary quadratic form, and the hyperbolic quadric with canonical equation $X_0 X_1+X_2X_3+\cdots+ X_{n-1}X_n=0$.
For $n$ even, there is the parabolic quadric with canonical equation $X_0^2 + X_1X_2+\cdots+ X_{n-1}X_n=0$. For a parabolic quadric $Q_n$, there is an unique point in $\PG(n,2)\setminus Q_n$, called {\it the nucleus} of $Q_n$, such that all line through the nucleus is tangent to $Q_n$ (see \cite[page10]{HirT3}). Table \ref{table quadric} shows the number of points of different non-singular quadrics.

\begin{table}[h!]
\begin{tabular}{|c|c|c|c|}
 \hline
 quadric $Q_n$ & Elliptic & Hyperbolic & Parabolic \\
 \hline
  number of points & $2^n-2^{\frac{n-1}{2}}-1$ & $2^n+2^{\frac{n-1}{2}}-1$ & $2^n-1$\\
  \hline
\end{tabular}
\vspace{0.3cm}
\caption{Number of points in non-singular quadrics}\label{table quadric}
\end{table}
\vspace{-0.3cm}

A singular quadric in $\PG(n,2)$ is either an $m$-space, $m<n$, or a {\it cone} $\Pi_{n-t-1} Q_t$ which is the set of points on the lines joining an $(n-t-1)$-space $\Pi_{n-t-1}$ to a non-singular quadric $Q_t$ in a $t$-space $\Pi_t$ with $\Pi_{n-t-1} \cap \Pi_t = \emptyset$. The number of points of such a cone is
\begin{equation}\label{eqn cone}
| \Pi_{n-t-1} Q_t| = (2^{n-t}-1)+2^{n-t}|Q_t|.
\end{equation}

A {\it polarity} $\rho$ of $\PG(n,2)$ is an order-two bijection on its subspaces that reverses containment. That is, for an $m$-space $\Pi_m$ and $m'$-space $\Pi_{m'}$ of $\PG(n,2)$, if $\Pi_m \subset \Pi_{m'}$, then $\Pi_{m'}^\rho \subset \Pi_m^\rho$. In particular, a polarity interchanges $m$-spaces and $(n-1-m)$-spaces. For a general reference on polarities, see \cite[Section 2.1]{Hir1}.

The {\it (binary linear) code} $C(\Gamma)$ of a graph $\Gamma=(V,E)$ is the subspace in the vector space $\mathbb F_2^{|V|}$ generated by the rows of the adjacency matrix of $\Gamma$ modulo 2. The {\it length} $n$ of $C(\Gamma)$ is $|V|$, and the {\it dimension} $k$ of $C(\Gamma)$ is the dimension of $C(\Gamma)$ as a subspace in $\mathbb F_2^{|V|}$. For any vector $w=(w_x)_{x\in V}\in \mathbb F_2^{|V|}$, the {\it weight} $\wt(w)$ of $w$ is $$\wt(w)=|\{x\in V | w_x \neq 0\}|.$$ The {\it minimum weight} $d$ of a code is the minimum of the weight of its non-zero codewords. A binary linear code of length $n$, dimension $k$ and minimum weight $d$ will be referred to as an $[n,k,d]_2$. For any subset $U\subset V$, the {\it characteristic vector of $U$}, denoted by $v^U,$ is the vector $(w_x)_{x\in V}$ where $w_x=1$ if $x\in U$, and $w_x=0$ if $x\notin U$. For a general reference on codes, see \cite{AK}.

For the graph $\Gamma_Q=(V_\Q,E_\Q)$ defined in Section \ref{sc intro}, $C(\Gamma_{\Q})$ is a $[ 2^n+2^{\frac{n-1}{2}},n+1,2^{n-1}]_2$ code if $\Q$ is elliptic, and is a $[ 2^n+2^{\frac{n-1}{2}},n+1,2^{n-1}-2^{\frac{n-1}{2}}]_2$ code if $\Q$ is hyperbolic. A vector $w\in \mathbb F_2^{|V_\Q|}$ is a codeword of $C(\Gamma_{\Q})$ if and only if it is the characteristic vector of $(\PG(n,2)\setminus \Q ) \setminus \Sigma$ for some $(n-1)$-space $\Sigma$ in $\PG(n,2)$. The weight distribution of $C(\Gamma_{\Q})$ is shown in Tables \ref{table weightE} and \ref{table weightH} (see for example \cite{HPV}).

\begin{table}[h!]
\begin{tabular}{|c||c|c|c|}
 \hline
 weight & $0$ & $2^{n-1}$ & $2^{n-1}+ 2^{\frac{n-1}{2}}$ \\ \hline
 number of codewords & 1 & $2^{n}- 2^{\frac{n-1}{2}}-1$ & $2^{n}+ 2^{\frac{n-1}{2}}$ \\
 \hline
\end{tabular}
\vspace{0.3cm}
\caption{Weight distribution of $C(\Gamma_{\Q})$ if $\Q$ is elliptic}\label{table weightE}
\vspace{-0.3cm}
\end{table}

\begin{table}[h!]
\begin{tabular}{|c||c|c|c|}
 \hline
 weight & $0$ & $2^{n-1}- 2^{\frac{n-1}{2}}$ &$2^{n-1}$ \\ \hline
 number of codewords & 1 & $2^{n}- 2^{\frac{n-1}{2}}$ & $2^{n}+ 2^{\frac{n-1}{2}}-1$ \\
 \hline
\end{tabular}
\vspace{0.3cm}
\caption{Weight distribution of $C(\Gamma_{\Q})$ if $\Q$ is hyperbolic}\label{table weightH}
\vspace{-0.3cm}
\end{table}

\section{Two constructions of Godsil-McKay switching sets of $\Gamma_{\Q}$}\label{sc switching set}

In this section, we will prove Theorems \ref{thm A} and \ref{thm B}, which give constructions of Godsil-McKay switching sets of the graph $\Gamma_{\Q}$ defined in Section \ref{sc intro} for quadrics $\Q$ in $\PG(n,2)$.

Theorems \ref{thm A} and \ref{thm B} are as follows.

\begin{thmx}\label{thm A}
Let $\Q$ be a non-singular quadric in $\PG(n,2)$ where $n\geq 5$ is odd. Let $t$ be an integer such that $0 < t \leq \frac{n-3}{2}$, $\alpha$ be a $t$-space in $\Q$, and $\Pi$ be a $(t+1)$-space meeting $\Q$ in exactly $\alpha$. Let $\Gamma_{\Q}$ be as defined in Section \ref{sc intro}. Then
\begin{equation}
S_t:=\Pi \setminus \alpha
\end{equation}
is a Godsil-McKay switching set of $\Gamma_{\Q}$ of size $2^{t+1}$. Let $\Gamma_{\Q,t}$ be the graph obtained by Godsil-McKay switching with switching set $S_t$. Then $\Gamma_{\Q,t}$ is a strongly regular graph with the same parameters as $\Gamma_{\Q}$ (which are listed as in Table \ref{table para}). Furthermore, if $\perp$ is the polarity of $\PG(n,2)$ induced by $\Q$, then
\begin{equation}\label{eqn Ts1}
T_t := (\PG(n,2)\setminus \Q ) \setminus \alpha^\perp
\end{equation}
is the set of vertices in $\Gamma_{\Q}$ outside $S_t$ which have exactly $\frac{1}{2} |S_t|$ neighbours in $S_t$.
\end{thmx}

\begin{thmx}\label{thm B}
Let $\Q$ be a non-singular quadric in $\PG(n,2)$ where $n\geq 5$ is odd. If $\Q$ is elliptic, then let $t$ be an integer such that $0 < t \leq \frac{n-3}{2}$. If $\Q$ is hyperbolic, then let $t$ be an integer such that $0 < t \leq \frac{n-5}{2}$. In $\PG(n,2)$ where $n\geq 5$ is odd, let $\Q$ be a non-singular quadric. Let $\alpha$ be a $t$-space in $\Q$. Let $\Pi,\Pi'$ be distinct $(t+1)$-spaces meeting $\Q$ in exactly $\alpha$ such that the space spanned by $\Pi$ and $\Pi'$ meet $\Q$ in exactly $\alpha$. Let $\Gamma_{\Q}$ be as defined in Section \ref{sc intro}. Then
\begin{equation}
S_{t,t}:=(\Pi \cup \Pi') \setminus \alpha
\end{equation}
is a Godsil-McKay switching set of $\Gamma_{\Q}$. Let $\Gamma_{\Q,t,t}$ be the graph obtained by Godsil-McKay switching with switching set $S_{t,t}$. Then $\Gamma_{\Q,t,t}$ is a strongly regular graph with the same parameters as $\Gamma_{\Q}$ (these are listed as in Table \ref{table para}). Furthermore, if $\perp$ is the polarity of $\PG(n,2)$ induced by $\Q$, then
\begin{equation}\label{eqn Tss1}
T_{t,t}=T_t \cup [ ((\Pi^\perp \triangle \Pi'^\perp)\setminus S_{t,t} )\setminus \Q]
\end{equation}
is the set of vertices in $\Gamma_{\Q}$ outside $S_{t,t}$ which have exactly $|\frac{1}{2} S_{t,t}|$ neighbours in $S_{t,t}$, where $\triangle$ is the symmetric difference.
\end{thmx}

\begin{remark}
In both Theorems \ref{thm A} and \ref{thm B}, $t\leq\frac{n-3}{2}$ or $t\leq\frac{n-5}{2}$. This is a necessary and sufficient condition for the existence of $\alpha$, $\Pi$ and $\Pi'$ by \cite[Theorem 22.8.3]{HirT3}.
\end{remark}

\begin{remark}
In Theorem \ref{thm B}, by the dimension theorem for subspaces, the space $\left<\Pi,\Pi'\right>$ spanned by $\Pi$ and $\Pi'$ is an $(t+2)$-subspace. By \cite[Theorem 3.1]{Hir1}, there are exactly three planes through $\alpha$ in $\left<\Pi,\Pi'\right>$. Let $\Pi''$ be the plane through $\alpha$ other than $\Pi$ and $\Pi'$. Since $\Q$ is a quadric, either $\left<\Pi,\Pi'\right> \cap \Q = \Pi''$ or $\left<\Pi,\Pi'\right> \cap \Q = \alpha$ holds by \cite[Theorem 22.8.3]{HirT3}. In the former case, since $\left<\Pi,\Pi'\right>$ is one dimension higher than that of $\Pi''$, $(\Pi \cup \Pi') \setminus \alpha = \left<\Pi,\Pi'\right> \setminus \Pi''$ is a Godsil-McKay switching set of $\Gamma_{\Q}$ by Theorem \ref{thm A}. The latter case is treated in Theorem \ref{thm B}.
\end{remark}

Throughout this article, we will work under the assumptions of Theorems \ref{thm A} or \ref{thm B}. In particular, the symbols $n$,$\Q$, $\perp$, $t$, $\alpha$, $\Pi$, $\Pi'$, $\Gamma_{\Q}$, $\Gamma_{\Q,t}$, $\Gamma_{\Q,t,t}$, $S_t$, $S_{t,t}$, $T_t$ and $T_{t,t}$ are preserved as defined in Theorems \ref{thm A} or \ref{thm B}.

We first check the size of switching sets described in Theorems \ref{thm A} and \ref{thm B}.

\begin{lemma}\label{lemma size}
The size $|S_t|$ of $S_t$ and the size $|S_{t,t}|$ of $S_{t,t}$ are respectively $2^{t+1}$ and $2^{t+2}$.
\end{lemma}

\begin{proof}
Since $\Pi$ is a $(t+1)$-space and $\alpha$ is a $t$-space, $$|S_t|=|\Pi \setminus \alpha|=(2^{t+2}-1)-(2^{t+1}-1)=2^{t+1}.$$ For $S_{t,t}$, because of $\Pi \cap \Pi' = \alpha$, we have $$|S_{t,t}|=|(\Pi \cup \Pi') \setminus \alpha|=2(2^{t+2}-1)-(2^{t+1}-1)-(2^{t+1}-1)=2^{t+2}.$$
\end{proof}

We now determine the structure of the subgraphs of $\Gamma_{\Q}$ with vertex sets $S_t$ and $S_{t,t}$ respectively.

\begin{lemma}\label{lemma subgraph1}
The subgraph of $\Gamma_{\Q}$ with vertex set $S_t$ is null.
\end{lemma}

\begin{proof}
Since $\alpha$ is one dimension less than that of $\Pi$, a line in $\Pi$ either lies in $\alpha$ or is tangent to $\alpha$, and thus to $\Q$. In other words, no two vertices are joined in $\Gamma_{\Q}$.
\end{proof}

\begin{lemma}\label{lemma subgraph2}
The subgraph of $\Gamma_{\Q}$ with vertex set $S_{t,t}$ is a regular subgraph of degree $2^{t+1}$.
\end{lemma}

\begin{proof}
Let $x$ be a point in $S_{t,t}$. Without loss of generality, assume $x \in \Pi$. We count the number of neighbours of $x$. By the same argument used in Lemma \ref{lemma subgraph1}, $x$ is not adjacent to any vertex in $\Pi \setminus \alpha$.

Since the span of $\Pi$ and $\Pi'$ meets $\Q$ in exactly $\alpha$ by assumption, any line through $x$ and a point in $\Pi' \setminus \alpha$ is an external line of $\Q$. In other words, the vertex is adjacent to any vertex in $\Pi' \setminus \alpha$. Since the size of $\Pi' \setminus \alpha$ is $2^{t+1}$, the result follows.
\end{proof}

\begin{lemma}\label{lemma GMset}
$S_t$ is a Godsil-McKay switching set of $\Gamma_{\Q}$.
\end{lemma}

\begin{proof}
By Lemma \ref{lemma subgraph1}, the subgraph of $\Gamma_{\Q}$ with vertex set $S_t$ is null.

By Lemma \ref{lemma size}, $|S_t|=2^{t+1}$. Let $x$ be a point in $(\PG(n,2) \setminus \Q )\setminus S_{t}$. It suffices to show that either any line joining $x$ and a point of $S_t$ meets $\Q$, or there are exactly $2^t$ or $2^{t+1}$ points $y$ in $S_t$ such that the line $xy$ is an external line to $\Q$. Since any line in $\PG(n,2)$ has exactly three points, any line through two external points of $\Q$ is either a tangent or an external line. Thus, it also suffices to show that either any line joining $x$ and a point of $S_t$ is an external line, or there are exactly $2^{t+1}$, $2^t$ points $y$ in $S_t$ such that the line $xy$ is tangent to $\Q$.

Suppose the line $xy$ joining $x$ and a point $y$ of $S_t$ is tangent to $\Q$.
Since every line has only three points, the unique point $z$ on $xy$ other than $x$ and $y$ is a point of $\Q$. Let $\Sigma$ be the $(t+1)$-space spanned by $z$ and $\alpha$. Then $\Sigma$ meets $\Q$ in the $t$-space $\alpha$ and at least one point not in $\alpha$, namely $z$. By \cite[Theorem 22.8.3]{HirT3}, $\Sigma$ either lies in $\Q$ or meets $\Q$ in exactly two $t$-spaces.

If $\Sigma$ lies in $\Q$, then every line through $x$ and a point of $S_t$ is a tangent to $\Q$, and we are done.

If $\Sigma$ meets $\Q$ in exactly two $t$-spaces, say $\alpha$ and $\alpha'$, then $\alpha$ and $\alpha'$ meet in a ($t-1$)-space. Then a line $xy'$ through $x$ and a point $y'$ of $S_t$ is tangent to $\Q$ if and only if $y'$ is in $S_t \cap \left<x,\alpha'\right>$. Since
\[
\begin{split}
S_t \cap \left<x,\alpha'\right>
&=(\Pi \setminus \alpha )\cap \left<x,\alpha'\right>\\
&=(\Pi \cap \left<x,\alpha'\right>) \setminus (\alpha \cap \left<x,\alpha'\right> )\\
&=\left<y,\alpha \cap \alpha'\right> \setminus (\alpha \cap \alpha')\\
\end{split}
\]
has $(2^{t+1}-1)-(2^{t}-1)=2^{t}$ points, the result follows.
\end{proof}

We need to make use of a property of $\perp$ to prove the following two lemmas. Recall from \cite[Lemma 22.3.3]{HirT3} that, for any point $y\in \Q$, $y^\perp$ comprises the points on the tangents to $\Q$ at $y$ and the lines in $\Q$ through $y$; for any point $y\notin \Q$, $y^\perp$ consists of the points on the tangents to $\Q$ through $y$.

\begin{lemma}\label{lemma SinPiperp}
The following inclusions hold:
\begin{enumerate}
  \item $S_t\subset \Pi^\perp \subset \alpha^\perp$.
  \item $S_{t,t}\subset \Pi^\perp \triangle \Pi'^\perp \subset \alpha^\perp$.
\end{enumerate}
\end{lemma}

\begin{proof}
Since $\alpha$ is a subset of $\Pi$ and $\Pi'$, by the definition of a polarity,
$\Pi^\perp$ and $\Pi'^\perp$ are subsets of $\alpha^\perp$.

Since the line through any two points of $\Pi$ is either a tangent of $\Q$ or a line of $\Q$. By \cite[Lemma 22.3.3]{HirT3}, $\Pi$ is a subset of $\Pi^\perp$. Similarly, $\Pi'$ is a subset $\Pi'^\perp$. Hence $S_t \subset \Pi^\perp$ and $S_{t,t}\subset \Pi^\perp\cup\Pi'^\perp$.

As stated in Theorem \ref{thm B}, the space spanned by $\Pi$ and $\Pi'$ meets $\Q$ in exactly $\alpha$. Thus, any line through joining a point of $\Pi\setminus \alpha$ and a point of $\Pi'\setminus \alpha$ is an external line of $\Q$. By \cite[Lemma 22.3.3]{HirT3}, $\Pi \cap \Pi'^\perp=\emptyset$ and $\Pi^\perp \cap \Pi' =\emptyset$, and so $S_{t,t}\cap \Pi^\perp\cap\Pi'^\perp =\emptyset$. The result follows.
\end{proof}

To determine $T_t$, we prepare a lemma about polarities in $\PG(n,2)$.

\begin{lemma}\label{lemma polarity}
Let $\rho$ be a polarity of $\PG(n,2)$. Let $\Sigma$ be an $(m+1)$-space of $\PG(n,2)$ where $0\leq m<n-1$. Let $x$ be a point in $\PG(n,2)$. Then exactly one of the following cases occurs.
\begin{enumerate}
 \item $x$ is in $\Sigma^\rho$.
 \item $x$ is in $\pi^\rho \setminus \Sigma^\rho$ for exactly one $m$-space $\pi$ in $\Sigma$.
\end{enumerate}
\end{lemma}

\begin{proof}
By \cite[Theorem 3.1]{Hir1}, there are exactly $N=2^{m+2}-1$ $m$-spaces in $\Sigma$. Let $\pi_1, \pi_2, \cdots , \pi_{N}$ be the $m$-spaces contained in $\Sigma$. Since $\rho$ is a polarity, $\pi_i^\rho$, $i=1,2,\cdots,N$, are $(n-1-m)$-spaces containing the $(n-2-m)$-space $\Sigma^\rho$. For distinct $i,j\in\{1,2,\cdots,N\}$, $\pi_i^\rho \cap \pi_j^\rho = \left<\pi_i,\pi_j\right>^\rho=\Sigma^\rho$. Thus, the number of points in $\bigcup_{i=1}^{N}\pi_i^\rho$ is
$ \left| \Sigma^\rho \right|+\sum_{i=1}^{N} \left|\pi_i^\rho \setminus \Sigma^\rho \right|
=\ (2^{n-1-m}-1)+N [(2^{n-m}-1)- (2^{n-1-m}-1)]
=\ 2^{n+1}-1,$
which is the number of points in $\PG(n,2)$. Now, the result follows.
\end{proof}

\begin{lemma}\label{lemma partition}
Let $x$ be a point not in $\Q$.
Then exactly one of the following cases occurs.
\begin{enumerate}
 \item $x$ is in $\Pi^\perp$; any line joining $x$ and a point in $S_t$
is not an external line of $\Q$.
 \item $x$ is in $\pi^\perp\setminus \Pi^\perp $ for exactly one $t$-space $\pi\neq \alpha$ in $\Pi$; the line $xy$ through $x$ and a point $y \in S_t$ is an external line of $\Q$ if and only if $y \notin \pi $.
 Furthermore, there are $2^{t}$ such points $y$.
 \item $x$ is in $\alpha^\perp\setminus \Pi^\perp $; any line through a point of $S_t$ and $x$ is an external line of $\Q$.
\end{enumerate}
\end{lemma}

\begin{proof}
By \cite[Theorem 3.1]{Hir1}, there are exactly $N=2^{t+2}-1$ $t$-spaces in $\Pi$. Let $\pi_0,\pi_1, \cdots , \pi_{N-1}$ be the $t$-spaces contained in $\Pi$. Without loss of generality, assume $\pi_0=\alpha$.

Let $x$ be a point not in $\Q$. By Lemma \ref{lemma polarity}, $x$ is either in $\Pi^\perp$ or in $\pi_i^\perp \setminus \Pi^\perp $ for exactly one $i \in \{0,1,\cdots,N-1\}$.

\begin{enumerate}
 \item Suppose $x\in \Pi^\perp$. Then $x\in y^\perp$ for all $y \in \Pi$. Thus the line through $x$ and a point $y \in \Pi \setminus \Q$ is a tangent to $\Q$. In order words, no point $y$ in $\Pi \setminus \Q=\Pi \setminus \alpha=S_t$ satisfies the condition that the line $xy$ is an external line of $\Q$.
\item Suppose $x\in \pi_i^\perp \setminus \Pi^\perp $ for exactly one $i\neq 0$. By a similar argument, it follows that no point $y$ in $\pi_i \setminus \Q=\pi_i \setminus \alpha$ satisfies the condition that the line $xy$ is an external line of $\Q$. Suppose there exists $z\in S_t \setminus \pi_i$ such that the line through $xz$ is not an external line of $\Q$. Since every line contains exactly three points, that line is tangent to $\Q$ and thus $x$ is in $z^\perp$. Then $x \in z^\perp \cap \pi_i^\perp = \left<z, \pi_i\right>^\perp = \Pi^\perp$. This gives a contradiction, and thus the line through $x$ and a point $y \in S_t$ is an external line of $\Q$ if and only if $y \notin \pi_i$. Since $\alpha\cap \pi_i$ is a ($t-1$)-space, there are exactly
$$|S_t \setminus \pi_i |
=|(\Pi \setminus \pi_i)\setminus (\alpha \setminus \pi_i)|
=[(2^{t+2}-1)-(2^{t+1}-1)]-[(2^{t+1}-1)-(2^t-1)]
=2^t$$
points $y$ in $S_t$ such that the line $xy$ is an external line of $\Q$.

\item Suppose $x\in \alpha^\perp \setminus \Pi^\perp$. Suppose there exists $y\in S_t$ such that the line $xy$ is not an external line of $\Q$. Then that line is a tangent to $\Q$ and thus $x \in y^\perp$. Then $x \in y^\perp \cap \alpha^\perp = \left<y, \alpha\right>^\perp = \Pi^\perp$. This gives a contradiction and the result follows.
 \end{enumerate}
\end{proof}

We are ready to give a proof of Theorem \ref{thm A}.

\begin{proof}[{\bf Proof of Theorem \ref{thm A}}]
By Lemma \ref{lemma GMset}, $S_t$ is a Godsil-McKay switching set for $\Gamma_{\Q}$.

By Godsil and McKay \cite{GM}, $\Gamma_{\Q,t}$ has a same adjacency spectrum as $\Gamma_{\Q}$. Since $\Gamma_{\Q}$ is a strongly regular graph, $\Gamma_{\Q,t}$ is also a strongly regular graph with the same parameters (see the first three paragraphs on \cite[Subsection 14.5.1]{BH}), where the parameters are listed as in Table \ref{table para} on \pageref{table para}.

By the definition of $T_t$ and Lemma \ref{lemma partition}, $$T_t =(\PG(n,2)\setminus \Q)\cap [\big( \bigcup_{\pi \neq \alpha}\pi^\perp \setminus \Pi^\perp \big) \setminus S_t]$$ where $\pi$ runs over all $t$-space of $\Pi$ except $\alpha$. By Lemma \ref{lemma polarity}, $\big( \bigcup_{\pi \neq \alpha}\pi^\perp \setminus \Pi^\perp \big)=\PG(n,2)\setminus \alpha^\perp$. Since $S_t$ is in $\alpha^\perp$, the result follows.
\end{proof}

With Lemma \ref{lemma partition}, we prove Theorem \ref{thm B}.

\begin{proof}[{\bf Proof of Theorem \ref{thm B}}]
By Lemma \ref{lemma subgraph2}, the subgraph of $\Gamma_{\Q}$ with vertex set $S_{t,t}$ is a regular subgraph of degree $2^{t+1}$.

Let $x$ be a point in $(\PG(n,2) \setminus \Q )\setminus S_{t,t}$. By Lemma \ref{lemma polarity}, one of the following cases occurs.\\
(1) $x\in \Pi^\perp$ and $x\in \Pi'^\perp$.\\
(2) $x\in \Pi^\perp$ and $x\in \alpha^\perp \setminus \Pi'^\perp$.\\
(3) $x\in \Pi^\perp$ and $x\in \pi' \setminus \Pi'^\perp$ for some $t$-space $\pi'\neq \alpha$ of $\Pi'$.\\
(4) $x\in \alpha^\perp \setminus \Pi^\perp$ and $x\in \Pi'^\perp$.\\
(5) $x\in \alpha^\perp \setminus \Pi^\perp$ and $x\in \alpha^\perp \setminus \Pi'^\perp$.\\
(6) $x\in \pi^\perp \setminus \Pi^\perp$ for some $t$-space $\pi\neq \alpha$ of $\Pi$,
and $x\in \Pi'^\perp$.\\
(7) $x\in \pi^\perp \setminus \Pi^\perp$ for some $t$-space $\pi\neq \alpha$ of $\Pi$, and $x\in \pi'^\perp \setminus \Pi'^\perp$ for some $t$-space $\pi'\neq \alpha$ of $\Pi'$.



Note that case (3) never occurs. Indeed, since $\alpha$ is a subset of $\Pi$, we have $\Pi^\perp \subset \alpha^\perp$.
Indeed, if $x$ is in $\Pi^\perp$, then $x$ is in $\alpha^\perp$ by Lemma \ref{lemma SinPiperp}. By Lemma \ref{lemma polarity}, $\alpha^\perp = (\alpha^\perp \setminus \Pi'^\perp)\cup \Pi'^\perp$ is disjoint from $\pi'^\perp \setminus \Pi'^\perp$. Similarly, case (6) never occurs.

For the remaining cases, by Lemma \ref{lemma partition}, there are respectively
$0+0=0$,
$0+2^{t+1} =2^{t+1}$,
$2^{t+1}+0 =2^{t+1}$,
$2^{t+1}+2^{t+1}=2^{t+2}$,
$2^{t}+2^{t} =2^{t+1}$
points $y$ in $(\Pi \cup \Pi') \setminus \alpha$ such that the line $xy$ is an external line of $\Q$. Therefore, $S_{t,t}$ is a Godsil-McKay switching set of $\Gamma_{\Q}$ because we have $|S_{t,t}|=2^{t+2}$ by Lemma \ref{lemma size}.

Similarly, by Godsil and McKay \cite{GM}, $\Gamma_{\Q,t,t}$ has a same adjacency spectrum as $\Gamma_{\Q}$. By \cite{BH}, $\Gamma_{\Q,t,t}$ is also a strongly regular graph with the same parameters, where the parameters are listed as in Table \ref{table para}.

The vertex $x$ is adjacent to none or all vertices in $S_{t,t}$, if and only if case (1) or (5) holds, if and only if $x \in \alpha^\perp \setminus (\Pi^\perp \triangle \Pi'^\perp)$. The result for $T_{t,t}$ now follows.
\end{proof}

\section{Some codewords of the switched graphs}\label{sc codes}
We shall use the same notation $n$, $\Q$, $\perp$, $t$, $\alpha$, $\Pi$, $\Pi'$, $\Gamma_{\Q}$, $\Gamma_{\Q,t}$, $\Gamma_{\Q,t,t}$, $S_t$, $S_{t,t}$, $T_t$, $T_{t,t}$, as described in Theorems \ref{thm A} or \ref{thm B}. Recall from Section \ref{sc terminology} that $C(\Gamma_{\Q,t})$ and $C(\Gamma_{\Q,t,t})$ are respectively the code of $C(\Gamma_{\Q,t})$ and $C(\Gamma_{\Q,t,t})$. In this section, we aim to prove $v^{S_t}, v^{T_t} \in C(\Gamma_{\Q,t})$ and $v^{S_{t,t}}, v^{T_{t,t}} \in C(\Gamma_{\Q,t,t})$.

Since we will need frequently the number of external lines of a non-singular quadric through a point, we give these numbers in the following lemma for ease of reference.

\begin{lemma}\label{lemma linethox}
Let $Q_m$ be a non-singular quadric in $\PG(m,2)$. Let $x$ be a point not in $Q_m$.
If $m$ is odd, there are $2^{m-2} \pm 2^{\frac{m-3}{2}}$ external lines through $x$, where the upper sign of $\pm$ is taken when if $Q_m$ is elliptic, and otherwise if $Q_m$ is hyperbolic. If $m$ is even, there are $0$ or $2^{m-2}-1$ external lines through $x$, depending on whether $x$ is the nucleus of $Q_m$ or not.
\end{lemma}

\begin{proof}
When $m$ is odd, $Q_m$ has $2^m\mp 2^{\frac{m-1}{2}}-1$ points (see Table \ref{table para}). Thus, there are
$|\PG(m,2)|-|Q_m| =2^m\pm 2^{(m-1)/2}$ non-quadric points. By \cite[Theorem 22.6.6]{HirT3}, these non-quadric points are in the same orbit under the subgroup $\Aut(Q_m)$ of the automorphism group of $\PG(m,2)$ which fixes $Q_m$. Thus, through each point, there are a same number of external lines. The result follows because there are $\frac{1}{3}(2^{m-2}) (2^{\frac{m+1}{2}}\pm 1) (2^{\frac{m-1}{2}}\pm 1)$ external lines in $\PG(m,2)$ \cite[Lemma 22.8.1]{HirT3}.

Similarly, when $m$ is even, there are $2^m$ non-quadric points. Recall from Section \ref{sc terminology} that all line through the nucleus of $Q_n$ is tangent to $Q_n$. By \cite[Theorem 22.6.6]{HirT3}, any non-quadric points, other than the nucleus, are in the same orbit under $\Aut(Q_m)$. The result follows similarly because there are
$\frac{1}{3} (2^{m-2}) (2^m-1)$ external lines in $\PG(m,2)$ \cite[Lemma 22.8.1]{HirT3}.
\end{proof}

In the following lemma, whenever we use the signs $\pm$ or $\mp$, the upper sign is always taken when $\Q$ is elliptic, and lower sign is always taken when $\Q$ is hyperbolic.

\begin{lemma}\label{lemma extline1}
There is an external line $l$ of $\Q$ such that $l$ and $\alpha^\perp$ are disjoint.
\end{lemma}

\begin{proof}
Let $x$ be a non-quadric point not in $\alpha^\perp$. Let $\Sigma$ be the $(n-t)$-space spanned by $\{x, \alpha^\perp\}$. If there is an external line of $\Q$ through $x$ but not in $\Sigma$, then such a line will be disjoint from $\alpha^\perp$ and we are done.

We first consider the case for $t=1$. By Lemma \ref{lemma polarity}, $x \in \pi^\perp$ for a unique $1$-space $\pi$ of $\Pi$. Since $x$ is not in $\alpha^\perp$, we have $\pi \neq \alpha$ and so $\pi \cap \alpha$ is a point of $\Q$. By Theorem \cite[Theorem 22.7.2]{HirT3}, $\Sigma \cap \Q$ is a parabolic quadric. If $x$ is the nucleus of $\Sigma \cap \Q$, then there is no external line (of both $\Q$ and $\Sigma \cap \Q$) in $\Sigma$ and through $x$, as desired. If $x$ is not the nucleus of $\Sigma \cap \Q$, then there are $$2^{n-3}-1$$ external lines in $\Sigma$ and through $x$ by Lemma \ref{lemma linethox}. Since $n$ is not less than $5$, this number is less than the number of external lines in $\PG(n,2)$ through $x$ found in Lemma \ref{lemma linethox}, and thus there is an external line of $\Q$ through $x$ but not in $\Sigma$, as desired.

Similarly, in case $t=2$, $\Sigma$ is an $(n-2)$-space meeting $\Q$ in a line cone $\Pi_1\Q^-_{n-4}$ over an elliptic quadric $\Q^-_{n-4}$ if $\Q$ is elliptic, and a line cone $\Pi_1\Q^+_{n-4}$ over a hyperbolic quadric $\Q^+_{n-4}$ if $\Q$ is hyperbolic.
Since $\Pi_1\Q^-_{n-4}$ has
\begin{equation}\label{eqn coneE}
|\Pi_1\Q^-_{n-4}|=3+4(2^{n-4}-2^{\frac{n-5}{2}}-1)=2^{n-2}- 2^{(n-1)/2}-1
\end{equation}
points and $\Pi_1\Q^+_{n-4}$ has
\begin{equation}\label{eqn coneH}
|\Pi_1\Q^+_{n-4}|=3+4(2^{n-4}+2^{\frac{n-5}{2}}-1)=2^{n-2}+2^{(n-1)/2}-1
\end{equation}
points, there are
$$|\Sigma|-|\Pi_1\Q^\epsilon_{n-4}|
=2^{n-2}\pm 2^{\frac{n-1}{2}} , \epsilon\in\{-,+\}$$
non-quadric points in the $(n-2)$-space $\Sigma$. Thus, there are at most $$\frac{2^{n-2}\pm 2^{\frac{n-1}{2}}}{2}=2^{n-3}\pm 2^{\frac{n-3}{2}}$$ external lines in $\Sigma$ through $x$. Since this number is less than the number of external lines in $\PG(n,2)$ through $x$ found in Lemma \ref{lemma linethox}, there is an external line of $\Q$ through $x$ but not in $\Sigma$, as desired.

We now consider the case for $t>2$. By \cite[Theorem 3.1]{Hir1}, through $x$, there are $$2^{n-t}-1$$ lines in the $(n-t)$-space $\Sigma$. Since this number is less than the number of external lines through $x$ found in Lemma \ref{lemma linethox}, there is an external line of $\Q$ through $x$ but not in $\Sigma$, as desired.
\end{proof}

\begin{lemma}\label{lemma char S}
The vector $v^{S_t}$ is in $C(\Gamma_{\Q,t})$. The vector $v^{S_{t,t}}$ is in $C(\Gamma_{\Q,t,t})$.
\end{lemma}

\begin{proof}
Let $l=\{x_1,x_2,x_3\}$ be an external line of $\Q$ such that $l$ and $\alpha^\perp$ are disjoint. This exists by Lemma \ref{lemma extline1}.

For each $i=1,2,3$, let $r_i$, $\dot{r_i}$ and $\ddot{r_i}$ respectively be the row of the adjacency matrices of $\Gamma_{\Q}$, $\Gamma_{\Q,t}$, $\Gamma_{\Q,t,t}$ corresponding to $x_i$. Then $r_i$ is the characteristic vector of $(\PG(n,2)\setminus \Q ) \setminus x_i^\perp $. By Lemma \ref{lemma polarity}, $\PG(n,2)\setminus \Q$ is the disjoint union of $l^\perp \setminus \Q$, $(x_1^\perp \setminus l^\perp) \setminus \Q $, $(x_2^\perp \setminus l^\perp) \setminus \Q $ and $(x_3^\perp \setminus l^\perp) \setminus \Q $. Since $l^\perp$ is a subset of $x_i^\perp$ for $i=1,2,3$, we have
\begin{equation}\label{eqn ri1}
r_1+r_2+r_3=0
\end{equation}
in $\mathbb F_2^{|V_\Q|}$.

Since $l$ is disjoint from $\alpha^\perp$, we have $l\subset T_t $ and $l\subset T_{t,t}$. By the definitions of $\Gamma_{\Q,t}$ and $\Gamma_{\Q,t,t}$, for each $i=1,2,3$, we have
\begin{equation}\label{eqn dri1}
\dot{r_i}=r_i+v^{S_t}
\end{equation}
and
\begin{equation}\label{eqn ddri1}
\ddot{r_i}=r_i+v^{S_{t,t}}.
\end{equation}
By \eqref{eqn ri1} and \eqref{eqn dri1}, $\dot{r_1}+\dot{r_2}+\dot{r_3}=v^{S_t}$ and so $v^{S_t}$ is a codeword of $C(\Gamma_{\Q,t})$. Similarly, $v^{S_{t,t}}$ is a codeword of $C(\Gamma_{\Q,t,t})$ because $\ddot{r_1}+\ddot{r_2}+\ddot{r_3}=v^{S_{t,t}}$ by \eqref{eqn ri1} and \eqref{eqn ddri1}.
\end{proof}

The purpose and proof of following lemma are similar to those of Lemma \ref{lemma extline1}, and we apply this lemma to prove $v^{T_t} \in C(\Gamma_{\Q,t})$ and $v^{T_{t,t}} \in C(\Gamma_{\Q,t,t})$.

\begin{lemma}\label{lemma extline2}
Let $x$ be a non-quadric point in $\alpha^\perp$. Then there is an external line $l$ of $\Q$ through $x$ such that $l$ is tangent to $\alpha^\perp$ at $x$.
\end{lemma}

\begin{proof}
To prove the lemma, it suffices to show some of external line through $x$ does not lie in $\alpha^\perp$.

We first consider the case for $t=1$. Then $\alpha^\perp$ is an $(n-2)$-space. By \cite[Theorem 22.7.2]{HirT3}, $\alpha^\perp \cap \Q$ is a line cone $\Pi_1\Q_{n-4}^-$ over an elliptic quadric $\Q^-_{n-4}$ if $\Q$ is elliptic, and a line cone $\Pi_1\Q_{n-4}^+$ over a hyperbolic quadric $\Q^+_{n-4}$ if $\Q$ is hyperbolic. For either $\Q$ elliptic or hyperbolic, the set of points $y$'s in $\alpha^\perp \cap \Q$ such that the line $xy$ is tangent to $\alpha^\perp \cap \Q$ forms a line cone $\Pi_1\Q_{n-5}$ over a parabolic quadric $\Q_{n-5}$. Since $\Q_{n-5}$ has $2^{n-5}-1$ points \cite[Theorem 5.21]{Hir1}, there are
$$|\Pi_1\Q_{n-5}|=[3+4(2^{n-5}-1)]=2^{n-3}-1$$
tangents in $\alpha^\perp$ through $x$. Using \eqref{eqn coneE} and \eqref{eqn coneH}, there are
$$\frac{|\alpha^\perp \cap \Q|-|\Pi_1\Q_{n-5}|}{2}
=2^{n-4}\mp 2^{(n-3)/2}$$
secants in $\alpha^\perp$ through $x$. Since there are $2^{n-2}-1$ lines in $\alpha^\perp$ through $x$ \cite[Theorem 3.1]{Hir1}, there are
\begin{equation}\label{eqn externalline}
(2^{n-2}-1)-(2^{n-4}\mp 2^{(n-3)/2})-(2^{n-3}-1)
=2^{n-4}\pm 2^{(n-3)/2}
\end{equation}
external lines of $\Q$ in $\alpha^\perp$ through $x$, where the upper signs of $\pm$ and $\mp$ are taken if $\Q$ is elliptic and the lower sign if $\Q$ is hyperbolic. Since the number in \eqref{eqn externalline} is less than the number of external lines through $x$ found in Lemma \ref{lemma linethox}, there is an external line of $\Q$ through $x$ but not in $\alpha^\perp$, as desired.

We now consider the case for $t>1$. By \cite[Theorem 3.1]{Hir1}, through $x$, there are only $$2^{n-1-t}-1$$ lines of the $(n-1-t)$-space $\alpha^\perp$. Since this number is less than the number of external lines through $x$ found in Lemma \ref{lemma linethox}, there is an external line of $\Q$ through $x$ but not in $\alpha^\perp$, as desired.
\end{proof}

\begin{lemma}\label{lemma char T}
The vector $v^{T_t}$ is in $C(\Gamma_{\Q,t})$. The vector $v^{T_{t,t}}$ is in $C(\Gamma_{\Q,t,t})$.
\end{lemma}

\begin{proof}
Let $x_1\in S_t$. Note that $x_1 \in \alpha^\perp$. Take an external line $l=\{x_1,x_2,x_3\}$ of $\Q$ through $x$ such that $l$ is tangent to $\alpha^\perp$ at $x_1$. It exists by Lemma \ref{lemma extline2}.

For each $i=1,2,3$, let $r_i$, $\dot{r_i}$ and $\ddot{r_i}$ respectively be the row of the adjacency matrices of $\Gamma_{\Q}$, $\Gamma_{\Q,t}$, $\Gamma_{\Q,t,t}$ corresponding to $x_i$. By the same argument used in the proof of Lemma \ref{lemma char S}, we have
\begin{equation}\label{eqn ri2}
r_1+r_2+r_3=0.
\end{equation}
Because of $x_1\in S_t \subset S_{t,t}$, by the definitions of $\Gamma_{\Q,t}$ and $\Gamma_{\Q,t,t}$, we have
\begin{equation}\label{eqn dr1}
\dot{r_1}=r_1+v^{T_t}
\end{equation}
and
\begin{equation}\label{eqn ddr1}
\ddot{r_1}=r_1+v^{T_{t,t}}.
\end{equation}
Since $x_2,x_3$ are not in $\alpha^\perp$, they are in $T_t$ and $T_{t,t}$ by \eqref{eqn Ts1} and \eqref{eqn Tss1}. So, for $i=2,3$, we have
\begin{equation}\label{eqn dri2}
\dot{r_i}=r_i+v^{S_t}
\end{equation}
and
\begin{equation}\label{eqn ddri2}
\ddot{r_i}=r_i+v^{S_{t,t}}.
\end{equation}
By \eqref{eqn ri2}, \eqref{eqn dr1} and \eqref{eqn dri2}, $\dot{r_1}+\dot{r_2}+\dot{r_3}=v^{T_t}$ and so $v^{T_t}$ is a codeword of $C(\Gamma_{\Q,t})$. Similarly, $v^{T_{t,t}}$ is a codeword of $C(\Gamma_{\Q,t,t})$ because $\ddot{r_1}+\ddot{r_2}+\ddot{r_3}=v^{T_{t,t}}$ by \eqref{eqn ri2}, \eqref{eqn ddr1} and \eqref{eqn ddri2}.
\end{proof}

\section{The minimum word of $C(\Gamma_{\Q,t})$ and $C(\Gamma_{\Q,{t,t}})$}\label{sc rank}
\label{sc rank}
In this section, we use the same notation $n$, $\Q$, $\perp$, $t$, $\alpha$, $\Pi$, $\Pi'$, $\Gamma_{\Q}$, $\Gamma_{\Q,t}$, $\Gamma_{\Q,t,t}$, $S_t$, $S_{t,t}$, $T_t$ and $T_{t,t}$ as in Section \ref{sc codes}, except requiring $n \geq 7$.

Let
\begin{equation}
C_t=\left<C(\Gamma_{\Q,t}),v^{S_t},v^{T_t}\right>
\end{equation}
and
\begin{equation}
C_{t,t}=\left<C(\Gamma_{\Q,t,t}),v^{S_{t,t}},v^{T_{t,t}}\right>
\end{equation}
In this section, we aim to prove the minimum word of $C_t$ and $C_{t,t}$ are respectively $v^{S_t}$ and $v^{S_{t,t}}$. This will give the minimum word of $C(\Gamma_{\Q,t})$ and $C(\Gamma_{\Q,t,t})$ once we prove that $C_t = C(\Gamma_{\Q,t})$ and $C_{t,t}=C(\Gamma_{\Q,t,t})$ in the next section.

\begin{lemma}\label{lemma size ws}
Let $w\in C(\Gamma_{\Q})$. Then $\wt(w+v^{S_t})> 2^{t+1}$ and $\wt(w+v^{S_{t,t}})> 2^{t+2}$.
\end{lemma}

\begin{proof}
From Table \ref{table weightE}, if $\Q$ is elliptic, the weight $\mbox{wt}(w)$ of $w$ satisfies $\mbox{wt}(w) \geq 2^{n-1}.$ By Lemma \ref{lemma size}, $\mbox{wt}(v^{S_t}) = 2^{t+1}$ and $\mbox{wt}(v^{S_{t,t}})=2^{t+2}$. So,
$$\mbox{wt}(w+v^{S_t})\geq \mbox{wt}(w)-\mbox{wt}(v^{S_t})= 2^{n-1}-2^{t+1},$$
$$\mbox{wt}(w+v^{S_{t,t}})\geq \mbox{wt}(w)-\mbox{wt}(v^{S_{t,t}})= 2^{n-1}-2^{t+2}.$$
Since we have assumed $n\geq 7$ in this section and we have $t\leq\frac{n-3}{2}$ under the assumption in Theorems \ref{thm A} and \ref{thm B}, it is straightforward to verify that $\wt(w+v^{S_t})> 2^{t+1}$ and $\wt(w+v^{S_{t,t}})> 2^{t+2}$.

From Table \ref{table weightH}, if $\Q$ is hyperbolic, then $\mbox{wt}(w) \geq 2^{n-1}-2^{\frac{n-1}{2}}$. Similarly, since $n\geq 7$, it is straightforward to verify that $\wt(w+v^{S_t})> 2^{t+1}$ and $\wt(w+v^{S_{t,t}})> 2^{t+2}$ with $t$ in the range stated in Theorems \ref{thm A} and \ref{thm B}.
\end{proof}

For any subset $U$ of points of $\PG(n,2)$, denoted by $\widehat{U}$ the set $U \setminus \Q$. Recall that whenever we use the signs $\pm$ or $\mp$, the upper sign is always taken when $\Q$ is elliptic, and lower sign is always taken when $\Q$ is hyperbolic.

\begin{lemma}\label{lemma size hat}
\begin{enumerate}
 \item\label{item alpha} $| \w{\alpha^\perp}|= 2^{n-t-1} \pm 2^{\frac{n-1}{2}}$.
 \item\label{item A} Let $A=(\Pi^\perp\triangle \Pi'^\perp ) \setminus S_{t,t}$. Then
 $|\w{A}| = 2^{n-t-2} \pm 2^{\frac{n-1}{2}} - 2^{t+2}$.
 \item\label{item sigma}
 Let $\Sigma$ be an ($n-1$)-space. Then exactly one of the following holds:
 \begin{enumerate}
 \item\label{item sigmaP} $\Sigma \cap \Q =\Q_{n-1}$; $|\w{\Sigma}|= 2^{n-1}$.
 \item\label{item sigmaC} $\Sigma \cap \Q = \Pi_0 \Q_{n-2}$ where $\Q_{n-2}$ and $\Q$ are both elliptic or both hyperbolic; $|\w{\Sigma} |= 2^{n-1} \pm 2^{\frac{n-1}{2}}$.
 \end{enumerate}
\end{enumerate}
\end{lemma}

\begin{proof}
\begin{enumerate}
 \item Since $\alpha$ is in $\Q$, by \cite[Theorem 22.8.3]{HirT3}, $\alpha^\perp \cap \Q$ is a cone $\Pi_t Q_{n-2t-2}$ where $Q_{n-2t-2}$ is elliptic if $\Q$ is elliptic, and is hyperbolic otherwise. By \eqref{eqn cone} and Table \ref{table para}, we have
$$|\alpha^\perp \cap \Q|
=(2^{t+1}-1)+2^{t+1}(2^{n-2t-2}\mp2^{\frac{n-2t-3}{2}}-1)
=2^{n-t-1}\mp2^{\frac{n-1}{2}}-1.$$
Since $\alpha^\perp$ is an $(n-t-1)$-space, it follows that
\[
\begin{split}
|\widehat{\alpha^\perp}|
=&|\alpha^\perp|-|\alpha^\perp \cap \Q|\\
=&(2^{n-t}-1)-[2^{n-t-1}\mp2^{\frac{n-1}{2}}-1]
=2^{n-t-1}\pm2^{\frac{n-1}{2}}.
\end{split}
\]
 \item Similar to \eqref{item alpha}, we have
\[
\begin{split}
|\widehat{\Pi^\perp}|
=&|\Pi^\perp|-|\Pi^\perp \cap \Q|=|\Pi^\perp|-|\Pi_t Q_{n-2t-3}|\\
=&(2^{n-t-1}-1)-[(2^{t+1}-1)+2^{t+1}(2^{n-2t-3}-1)]
=2^{n-t-2}
\end{split}
\]
and
\[
\begin{split}
|\w{\left< \Pi, \Pi'\right>^\perp}|
=&|\left< \Pi, \Pi'\right>^\perp|-|(\left< \Pi, \Pi'\right>^\perp)\cap \Q|
=|\left< \Pi, \Pi'\right>^\perp|-|\Pi_t \Q_{n-2t-4} |\\
=&(2^{n-t-2}-1)-[(2^{t+1}-1)+2^{t+1}(2^{n-2t-4} \pm 2^{\frac{n-2t-5}{2}}-1)]\\
=&2^{n-t-3}\mp 2^{\frac{n-3}{2}}.
\end{split}
\]
where $\Q_{n-2t-4}$ is hyperbolic if $\Q$ is elliptic; $\Q_{n-2t-4}$ is elliptic if $\Q$ is hyperbolic. Recall from Lemma \ref{lemma SinPiperp}, $S_{t,t} \subset \Pi^\perp\triangle \Pi'^\perp $. Now using Lemma \ref{lemma size}, we deduce
$$|\w{A}|
=|\widehat{\Pi^\perp}|+|\widehat{\Pi'^\perp}|-2|\w{\left< \Pi, \Pi'\right>^\perp}|-|S_{t,t}|
=2^{n-t-2} \pm 2^{\frac{n-1}{2}} - 2^{t+2}.$$
\item
 By \cite[Theorem 22.8.5]{HirT3}, $\Sigma \cap \Q$ is either (a) $\Q_{n-1}$ or (b) $\Pi_0 \Q_{n-2}$ where $\Q_{n-2}$ and $\Q$ are both elliptic or both hyperbolic. The result follows by \eqref{eqn cone} and Table \ref{table para}.
 \end{enumerate}
\end{proof}

\begin{lemma}\label{lemma size T}
The size of $T_t$ and $T_{t,t}$ are respectively $| T_t |=2^n-2^{n-t-1}$ and $| T_{t,t} |=2^n\pm 2^{\frac{n-1}{2}}-2^{n-t-2}-2^{t+2}$. Furthermore, the following holds:
\begin{enumerate}
 \item $| T_t |>2^{t+1}$.
 \item $| T_t \triangle S_t|>2^{t+1}$.
 \item $| T_{t,t}|>2^{t+2}$.
 \item $| T_{t,t}\triangle S_{t,t}|>2^{t+2}$.
\end{enumerate}
\end{lemma}

\begin{proof}
Using \eqref{eqn Ts1} and Lemma \ref{lemma size hat}\eqref{item alpha}, we obtain
$$|T_t|= | \PG(n,2)|-|\Q |-| \w{\alpha^\perp}|
= (2^{n+1}-1)-(2^n\mp2^{\frac{n-1}{2}}-1)-(2^{n-t-1} \pm 2^{\frac{n-1}{2}})
=2^n-2^{n-t-1}.$$
Since $0<t\leq\frac{n-3}{2}$, we have
$$|T_t|-2^{t+1}
=2^{n-t-1}(2^{t+1}-1)-2^{t+1}
> 3\cdot 2^{n-t-1}-2^{t+1}>0.$$
So, $|T_t|>2^{t+1}$.

Using \eqref{eqn Tss1} and Lemma \ref{lemma size hat}\eqref{item A}, we have
$$|T_{t,t}| = |T_t|+|\w{A}|
=2^n-2^{n-t-2} \pm 2^{\frac{n-1}{2}}-2^{t+2}$$
where $A=(\Pi^\perp\triangle \Pi'^\perp ) \setminus S_{t,t}$. Because of $t>0$, we have
$$|T_{t,t} |- 2^{t+2}
= 2^{n-t-2}(2^{t+2}-1)\pm 2^{\frac{n-1}{2}}-2^{t+3}
\geq 7 \cdot 2^{n-t-2} \pm 2^{\frac{n-1}{2}}-2^{t+3}.$$
When $\Q$ is elliptic, $t\leq \frac{n-3}{2}$ and so $$7 \cdot 2^{n-t-2}+2^{\frac{n-1}{2}} -2^{t+3}>0.$$
When $\Q$ is hyperbolic, $t\leq \frac{n-5}{2}$ and so
$$7 \cdot 2^{n-t-2}-2^{\frac{n-1}{2}} -2^{t+3} \geq 7 \cdot 2^{\frac{n-1}{2}} -2^{\frac{n-1}{2}} -2^{t+3}>0.$$
In both cases, $|T_{t,t} |> 2^{t+2}$. The results of $T_{t} \triangle S_{t}$ and $T_{t,t} \triangle S_{t,t}$ follow because of
$T_{t} \cap S_{t}=\emptyset$ and $T_{t,t} \cap S_{t,t}=\emptyset$.
\end{proof}

\begin{lemma}\label{lemma size RTS}
Let $R=(\PG(n,2)\setminus \Q ) \setminus \Sigma$ for some $(n-1)$-space $\Sigma$ of $\PG(n,2)$.
Then the following holds:
\begin{enumerate}
 \item $|R \triangle T_t| >2^{t+1}$.
 \item $|R \triangle T_t \triangle S_t|>2^{t+1}$.
 \item $|R \triangle T_{t,t}|>2^{t+2}$.
 \item $|R \triangle T_{t,t}\triangle S_{t,t}|>2^{t+2}$.
\end{enumerate}
\end{lemma}

\begin{proof}
The complement $R^c$ of $R$ in $\PG(n,2)\setminus \Q$ is
\[
R^c=\w{\Sigma}.
\]
Let $A:= ((\Pi^\perp \triangle \Pi'^\perp)\setminus S_{t,t} )\setminus \Q$. By \eqref{eqn Ts1} and \eqref{eqn Tss1}, we have
\begin{equation}
\begin{split}
{T_t}^c &= \w{\alpha^\perp }, \\
T_{t,t} &=T_t \cup A.
\end{split}
\end{equation}
Recall for any subsets $U_1$, $U_2$, $U_3$ of $\PG(n,2)\setminus \Q$,
we have
$U_1 \triangle U_2 = U_1^c \triangle U_2^c$; $(U_1 \cup U_2)^c = U_1^c \cap U_2^c$; $(U_1 \triangle U_2 ) \triangle U_3 = U_1 \triangle (U_2 \triangle U_3)$; $U_1 \triangle U_2 \supset U_1 \setminus U_2$, and equality holds if and only if
 $U_1 \subset U_2$.
Further because of $S_{t}, S_{t,t} \subset \alpha^\perp$ by Lemma \ref{lemma SinPiperp} and $S_{t,t}\cap A =\emptyset$, we have
\begin{equation}\label{eqn 1}
\begin{split}
R \triangle T_t&
=\w{\Sigma} \triangle \w{\alpha^\perp}
\supset \w{\Sigma} \setminus \w{\alpha^\perp};\\
R \triangle T_t \triangle S_t&
=(\w{\Sigma} \triangle \w{\alpha^\perp} )\triangle S_t
=\w{\Sigma} \triangle (\w{\alpha^\perp} \setminus S_t )
\supset\w{\Sigma} \setminus (\w{\alpha^\perp} \setminus S_t )
\supset\w{\Sigma} \setminus \w{\alpha^\perp}
;\\
R \triangle T_{t,t}&
=\w{\Sigma} \triangle (\w{\alpha^\perp} \cap \w{A^c})
=\w{\Sigma} \triangle (\w{\alpha^\perp} \setminus \w{A})
\supset \w{\Sigma} \setminus (\w{\alpha^\perp} \setminus \w{A});\\
R \triangle T_{t,t}\triangle S_{t,t}&
=\w{\Sigma} \triangle [(\w{\alpha^\perp} \setminus \w{A})\setminus S_{t,t}]
\supset \w{\Sigma} \setminus [(\w{\alpha^\perp} \setminus \w{A})\setminus S_{t,t}]
\supset \w{\Sigma} \setminus (\w{\alpha^\perp} \setminus \w{A}).
\end{split}
\end{equation}
Thus, it suffices to show (i) $|\w{\Sigma} \setminus \w{\alpha^\perp} |>2^{t+1}$
and (ii) $|\w{\Sigma} \setminus (\w{\alpha^\perp} \setminus \w{A})|>2^{t+2}$ for $t$ within the range mentioned in Theorems \ref{thm A} and \ref{thm B}. By \cite[Theorem 22.8.3]{HirT3}, $\Sigma \cap \Q$ is either (a) a parabolic quadric $\Q_{n-1}$, or (b) a point cone $\Pi_0 \Q_{n-2}$ where $\Q_{n-2}$ and $\Q$ are both elliptic or both hyperbolic.

\begin{enumerate}
\item[(a)]
If $\Sigma\cap \Q = Q_{n-1}$, then by \cite[Theorem 22.7.2]{HirT3}, we have $\Sigma^\perp \notin \Q$ and so $\Sigma^\perp \notin \alpha$. By the definition of a polarity, we have $\alpha^\perp \not\subset \Sigma$. Since $\Sigma$ is a hyperplane and $\alpha^\perp$ is an ($n-1-t$)-space, $\Sigma \cap \alpha^\perp$ is a ($n-2-t$)-space.

\begin{enumerate}
 \item[(i)]
By Lemma \ref{lemma size hat}\eqref{item sigmaP} and $0<t\leq \frac{n-3}{2}$, we have
\[
\begin{split}
|\w{\Sigma} \setminus \w{\alpha^\perp}|-2^{t+1}
\geq& |\w{\Sigma}|-|\Sigma \cap\alpha^\perp|-2^{t+1}\\
=& 2^{n-1}-(2^{n-t-1}-1)-2^{t+1}
= 2^{n-t-1}(2^t-1) +1-2^{t+1}\\
\geq&2^{n-t-1}-2^{t+1}+1
>0.
\end{split}
\]
 \item[(ii)]
Similarly, since $\w{\Sigma} \setminus (\w{\alpha^\perp}\setminus \w{A}) \subset \w{\Sigma} \setminus \w{\alpha^\perp}$, we have
\[
\begin{split}
|\w{\Sigma} \setminus (\w{\alpha^\perp}\setminus \w{A})|-2^{t+2}
\geq& |\w{\Sigma}|-|\Sigma \cap\alpha^\perp|-2^{t+2}\\
\geq&2^{n-t-1}-2^{t+2}+1
>0.
\end{split}
\]
\end{enumerate}
\item[(b)]
\begin{enumerate}
\item[(i)]
If $\Sigma\cap \Q=\Pi_0 \Q_{n-2}$, then by Lemma \ref{lemma size hat}(\ref{item alpha},\ref{item sigmaC}) and because of $t>0$, we have
\[
\begin{split}
|\w{\Sigma}\setminus\w{\alpha^\perp}|-2^{t+1}
\geq& |\w{\Sigma}|-|\w{\alpha^\perp}|-2^{t+1}\\
=& (2^{n-1} \pm 2^{\frac{n-1}{2}})-(2^{n-t-1} \pm 2^{\frac{n-1}{2}})-2^{t+1}\\
=& 2^{n-1-t}(2^{t}-1)-2^{t+1}
\geq 2^{n-1-t}-2^{t+1} > 0.
\end{split}
\]

\item[(ii)]
If $\Sigma\cap \Q=\Pi_0 \Q_{n-2}$, then by Lemma \ref{lemma size hat}(\ref{item alpha},\ref{item A},\ref{item sigmaC}) and because of $t>0$, we have
\begin{equation}\label{eqn b}
\begin{split}
\hspace{2cm}
&|\w{\Sigma} \setminus (\w{\alpha^\perp} \setminus \w{A})|-2^{t+2}\\
\geq& |\w{\Sigma}|-|\w{\alpha^\perp}|+|\w{A}|-2^{t+2}\\
=& (2^{n-1} \pm 2^{\frac{n-1}{2}})-(2^{n-t-1} \pm 2^{\frac{n-1}{2}})
+(2^{n-t-2} \pm 2^{\frac{n-1}{2}} - 2^{t+2})-2^{t+2}\\
=& 2^{n-2-t}(2^{t+1}-1) \pm 2^{\frac{n-1}{2}} - 2^{t+3}\\
\geq& 3\cdot 2^{n-2-t} \pm 2^{\frac{n-1}{2}} - 2^{t+3}
\end{split}
\end{equation}
where the last equality holds if and only if $t=1$.

If $\Q$ is elliptic, then because of $t\leq \frac{n-3}{2}$, we have
\begin{equation}\label{eqn c}
3\cdot 2^{n-2-t} + 2^{\frac{n-1}{2}} - 2^{t+3}\geq 0
\end{equation}
where the equality holds if and only if $t=\frac{n-3}{2}$. Because of $n\geq 7$, it is impossible to have $1=t=\frac{n-3}{2}$. Combining \eqref{eqn b} and \eqref{eqn c}, we have $|\w{\Sigma} \setminus (\w{\alpha^\perp} \setminus \w{A})|>2^{t+2}$.

If $\Q$ is hyperbolic, then because of $t\leq \frac{n-5}{2}$, we have
\begin{equation}\label{eqn d}
3\cdot 2^{n-2-t} - 2^{\frac{n-1}{2}} - 2^{t+3}> 0.
\end{equation}
Combining \eqref{eqn b} and \eqref{eqn d}, we have $|\w{\Sigma} \setminus (\w{\alpha^\perp} \setminus \w{A})|>2^{t+2}$.
\end{enumerate}
\end{enumerate}
\end{proof}

\begin{proposition}\label{prop char1}
Let $u$ be a non-zero vector in $C_t$. Then $\wt(u) \geq 2^{t+1}$, and equality holds if and only if $u=v^{S_{t}}$.
\end{proposition}

\begin{proof}
Let $u$ be a non-zero vector in $C_t$. Then $u$ is one of the following:
$w$, $w+ v^{S_t}$, $w + v^{T_t}$, $w + v^{T_t} + v^{S_t}$, $v^{T_t}$, $w^{T_t} + v^{S_t}$ or $v^{S_t}$ for some $w\in C(\Gamma_{\Q})$. By Tables \ref{table weightE} and \ref{table weightH}, $\wt(w) > 2^{t+1}$, and by Lemma \ref{lemma size ws}, $\wt(w+ v^{S_t}) > 2^{t+1}$. Note that for any subsets $U_1$, $U_2$ of $\PG(n,2)\setminus \Q$, $v^{U_1}+v^{U_2}=v^{U_1 \triangle U_2}$. The result follows from Lemmas \ref{lemma size}, \ref{lemma size T} and \ref{lemma size RTS} because $w=v^R$ where $R=(\PG(n,2)\setminus \Q ) \setminus \Sigma$ for some $(n-1)$-space $\Sigma$.
\end{proof}

\begin{proposition}\label{prop char2}
Let $u$ be a non-zero vector in $C_{t,t}$. Then $\wt(u) \geq 2^{t+2}$, and equality holds if and only if $u=v^{S_{t,t}}$.
\end{proposition}

\begin{proof}
It follows using arguments that are similar to those in the proof of Proposition \ref{prop char1}.
\end{proof}

\section{Numbers of switched graphs found}\label{sc number}

With the notation as given in Section \ref{sc rank} for $n$, $\Q$, $\perp$, $t$, $\alpha$, $\Pi$, $\Pi'$, $\Gamma_{\Q}$, $\Gamma_{\Q,t}$, $\Gamma_{\Q,t,t}$, $S_t$, $S_{t,t}$, $T_t$ and $T_{t,t}$, we assume $n \geq 7$. In this section, we will prove $C(\Gamma_{\Q,t})=C_t$ and $C(\Gamma_{\Q,t,t})=C_{t,t}$ as claimed in Section \ref{sc rank}, and then count the number of non-isomorphic graphs constructed through Theorems \ref{thm A} and \ref{thm B}.

Let $A,A_t,A_{t,t}$ be the adjacency matrices of $\Gamma_{\Q}$, $\Gamma_{\Q,t}$ and $\Gamma_{\Q,t,t}$.

Since $S_t \subset S_{t,t}$ and $T_{t} \subset T_{t,t}$, we may assume that the first $|S_t|$ rows and columns of $A,A_t,A_{t,t}$ correspond to points of $\PG(n,2)\setminus \Q$ in $S_t$; the next $|S_{t,t}\setminus S_{t}|$ rows and columns correspond to those in $S_{t,t}\setminus S_t$; the last $|T_{t,t}|$ rows and columns correspond to points in $T_{t,t}$ such that the last $|T_{t}|$ rows and columns correspond to points in $T_{t}$. By the definition of $\Gamma_{\Q,t}$,
\begin{equation}\label{eqn AAK}
A_t = A + M_t,\mbox{where }M_t =
\left(
\begin{array}{ccc}
O & O & J_t \\
O & O & O \\
J_t' & O & O
\end{array}
\right)
\end{equation}
where $J_t$ is the $|S_t|$-by-$|T_t|$ all-ones matrix. Similarly,
by the definition of $\Gamma_{\Q,t,t}$,
\begin{equation}\label{eqn AAK2}
\ A_{t,t} = A + M_{t,t}, \mbox{where }M_{t,t}=
\left(
\begin{array}{ccc}
O & O & J_{t,t} \\
O & O & O \\
J_{t,t}' & O & O
\end{array}
\right)
\end{equation}
where $J_{t,t}$ is the $|S_{t,t}|$-by-$|T_{t,t}|$ all-ones matrix.

\begin{lemma}\label{lemma TnotinC}
None of $v^{T_t}$ or $v^{T_{t,t}}$ is in $C(\Gamma_{\Q})$.
\end{lemma}

\begin{proof}
Suppose $v^{T_t}$ is in $C(\Gamma_{\Q})$. Recall any codeword in $C(\Gamma_{\Q})$ is $v^R$ where $R=(\PG(n,2)\setminus \Q ) \setminus \Sigma$ for some $(n-1)$-space $\Sigma$. By \eqref{eqn Ts1}, $$(\PG(n,2)\setminus \Q ) \setminus \alpha^\perp = (\PG(n,2)\setminus \Q ) \setminus \Sigma.$$ This implies $\Sigma \setminus \Q = \alpha^\perp\setminus \Q$. Considering the size of $\Sigma \setminus \Q$ and $ \alpha^\perp\setminus \Q$ given in Lemma \ref{lemma size hat}, we have $n=3$ or $t=0$, which contradicts the range of $n$ and $t$ stated in Theorem \ref{thm A} or Theorem \ref{thm B}.
\end{proof}

We now prove $C(\Gamma_{\Q,t})=C_t$ and $C(\Gamma_{\Q,t,t})=C_{t,t}$ as announced in Section \ref{sc rank}.

\begin{lemma}\label{lemma basis1}
$C(\Gamma_{\Q,t})=\left<C(\Gamma_{\Q,t}),v^{S_t},v^{T_t}\right>$ and the 2-rank of $C(\Gamma_{\Q,t})$ is $n+3$.
\end{lemma}

\begin{proof}
By Lemmas \ref{lemma char S} and \ref{lemma char T}, $v^{S_t}$ and $v^{T_{t}}$ are codewords of $C(\Gamma_{\Q,t})$. By \eqref{eqn AAK}, a row of the adjacency matrix of $\Gamma_{\Q,t}$ either is a row of the adjacency matrix of $\Gamma_{\Q}$ or differs from such a row by $v^{S_t}$ or $v^{T_t}$. Thus, any row of the adjacency matrix of $\Gamma_{\Q}$ is a codeword of $C(\Gamma_{\Q,t})$.

By Lemma \ref{lemma TnotinC}, $v^{T_t} \notin C(\Gamma_{\Q})$ and by Proposition \ref{prop char1}, for any $w\in C(\Gamma_{\Q})$, we have that none of $w$ and $w+v^{T_t}$ is the vector $v^{S_t}$. Thus, $v^{S_t}$, $v^{T_t}$ and a basis of $C(\Gamma_{\Q})$ form a linearly independent set of size $2+(n+1)=n+3$.

In \eqref{eqn AAK}, since the 2-rank of $M_t$ is $2$, the 2-rank of $C(\Gamma_{\Q})$ differs from that of $C(\Gamma_{\Q,t})$ by at most $2$. Since $v^{S_t}$ and $v^{T_t}$ and a basis of $C(\Gamma_{\Q})$ form a linearly independent set in $C(\Gamma_{\Q,t})$ with size two more than the 2-rank of $C(\Gamma_{\Q})$, they form a basis of $C(\Gamma_{\Q,t})$.
\end{proof}

\begin{lemma}\label{lemma basis2}
$C(\Gamma_{\Q,t,t})=\left<C(\Gamma_{\Q,t,t}),v^{S_{t,t}},v^{T_{t,t}}\right>$ and the 2-rank of $C(\Gamma_{\Q,t,t})$ is $n+3$.
\end{lemma}

\begin{proof}
The proof is similar to that of Lemma \ref{lemma basis1}.
\end{proof}

We now give the parameters of $C(\Gamma_{\Q,t})$ and $C(\Gamma_{\Q,t,t})$. Recall the upper sign of $\mp$ is taken when if $\Q$ is elliptic, and otherwise if $\Q$ is hyperbolic.

\begin{theorem}\label{thm code}
$C(\Gamma_{\Q,t})$ is a $[2^n \mp 2^{\frac{n-1}{2}},n+3,2^{t+1}]_2$-code. $C(\Gamma_{\Q,t,t})$ is a $[2^n \mp 2^{\frac{n-1}{2}},n+3,2^{t+2}]_2$-code
\end{theorem}

\begin{proof}
The length of $C(\Gamma_{\Q,t})$ and $C(\Gamma_{\Q,t,t})$ are the number of vertices of their respective graphs, which is $2^n \mp 2^{\frac{n-1}{2}}$. Other parameters of the codes follow from Lemmas \ref{lemma basis1}, \ref{lemma basis2}, and Proportions \ref{prop char1}, \ref{prop char2}.
\end{proof}

\begin{theorem}\label{thm count}
The graphs $\Gamma_\Q, \Gamma_{\Q,1}, \Gamma_{\Q,2}, \cdots, \Gamma_{\Q,\frac{n-3}{2}},\Gamma_{\Q,1,1}$, $\Gamma_{\Q,2,2}$, $\cdots$, $\Gamma_{\Q,m,m}$ are distinct up to isomorphism, where $m=\frac{n-3}{2}$ if $\Q$ is elliptic and $m=\frac{n-5}{2}$ if $\Q$ is hyperbolic.
\end{theorem}

\begin{proof}
$\Gamma_\Q$ is distinct from other graphs in the list because it has a 2-rank $n+1$ \cite[Theorem 5.3]{P} but others do not by Lemmas \ref{lemma basis1} and \ref{lemma basis2}. Let $\Gamma,\Gamma'$ be two graphs listed above other than $\Gamma_\Q$. Let $S$ and $S'$ be switching sets of $\Gamma_{\Q}$ such that $\Gamma,\Gamma'$ are obtained from $\Gamma_{\Q}$ with switching sets respectively $S$ and $S'$.

Suppose there is an isomorphism $\phi$ between $\Gamma$ and $\Gamma'$. Then $\phi$ induces a code isomorphism $\Phi$ between $C(\Gamma)$ and $C(\Gamma')$. Since $\Phi$ maps minimum word(s) of $C(\Gamma)$ to those of $C(\Gamma')$, we have $\Phi(S)=S'$ by Propositions \ref{prop char1} and \ref{prop char2}. Considering the size of the switching sets given in Lemma \ref{lemma size}, we may assume without loss of generality that $\Gamma=\Gamma_{\Q,t+1}$ and $\Gamma'=\Gamma_{\Q,t,t}$ for some $t$. By Lemma \ref{lemma subgraph1}, the subgraph of $\Gamma_{\Q}$, and hence of $\Gamma$, with vertex set $S=S_{t+1}$ is null. But by Lemma \ref{lemma subgraph2}, the subgraph of $\Gamma_{\Q}$, and hence of $\Gamma'$, with vertex set $S'=S_{t,t}$ is not null. This contradicts $\Phi(S)=S'$, and so $\Gamma$ and $\Gamma'$ are non-isomorphic.
\end{proof}

Since we work under the assumption that $n\geq 7$ in Sections \ref{sc rank} and \ref{sc number}, Theorem \ref{thm count} is valid under the same assumption. However, it can be checked directly that in case $\Q$ is elliptic and $n=5$, $\Gamma_{\Q}$, $\Gamma_{\Q,1}$ and $\Gamma_{\Q,1,1}$ are non-isomorphic; in case $\Q$ is hyperbolic and $n=5$, $\Gamma_{\Q}$ and $\Gamma_{\Q,1}$ are non-isomorphic. In conclusion, for $n\geq 5$, if $\Q$ is an elliptic quadric in $\PG(n,2)$, then Theorems \ref{thm A} and \ref{thm B} give $n-3$ non-isomorphic graphs, other than $\Gamma_{\Q}$, with the same parameters as $\Gamma_{\Q}$, where the parameters are shown in Table \ref{table para}.
For $n\geq 5$, if $\Q$ is a hyperbolic quadric in $\PG(n,2)$, then Theorems \ref{thm A} and \ref{thm B} give $n-2$ non-isomorphic graphs, other than $\Gamma_{\Q}$, with the same parameters as $\Gamma_{\Q}$.


\begin{thebibliography}{00}\footnotesize
\bibitem{AH}
Abiad, A. \& Haemers, W.H.
Switched symplectic graphs and their 2-ranks.
Appeared online in {\it Des. Codes Cryptogr.}

\bibitem{AK}
Assmus, E.F. \& Key, J.D. (1992).
{\it Designs and their codes}.
Cambridge University Press, Cambridge.

\bibitem{BJP}
Barwick, S.G., Jackson, W-A. \& Penttila, T.
New families of strongly regular graphs.
Manuscript.



\bibitem{BH}
Brouwer, A.E. \& Haemers, W.H. (2011).
{\it Spectra of graphs}.
Springer, New York.

\bibitem{CV}
Cameron, P.J. \& van Lint, J.H. (1991).
{\it Designs, graphs, codes and their links},
Cambridge University Press, Cambridge.

\bibitem{GM}
Godsil, C.D. \& McKay, B.D. (1982).
Constructing cospectral graphs.
{\it Aequationes Math}, {\bf 25}, 257--268.

\bibitem{HPV}
Haemers, W.H., Peeters, M.J.P. \& van Rijckevorsel, J.M. (1999).
Binary codes of strongly regular graphs.
{\it Des. Codes Cryptogr.}, {\bf 17}, 187--209.

\bibitem{HS}
Hall, J.I. \& Shult, E.E. (1985).
Locally cotriangular graphs.
{\it Geometriae Dedicata}, {\bf 18}, {113--159}.


\bibitem{Hir1}
Hirschfeld, J.W.P. (1979).
{\it Projective geometries over finite fields}.
Oxford Univ. Press.


\bibitem{HirT3}
Hirschfeld, J.W.P. \& Thas, J.A. (1991).
{\it General Galois geometries}.
Oxford Sci. Publ., Clarendon Press.

\bibitem{P}
Peeters, M.J.P. (1995).
Uniqueness of strongly regular graphs having minimal $p$-rank.
{\it Linear Algebra Appl.} {\bf 226--228}, 9--31.


\end{thebibliography}
\end{document}